\documentclass{amsart}
\usepackage{amsmath,amssymb,amsthm,amsfonts,enumerate,array,color,lscape,fancyhdr,layout,pst-all}
\usepackage[english]{babel}
\usepackage[latin1]{inputenc}
\usepackage{young}

\usepackage[hcentermath]{youngtab}
\usepackage{graphicx}
\usepackage{float}
\usepackage{pstricks}

\newcounter{numberofremark}
\newcommand\nothing[1]{}
\usepackage[active]{srcltx}
\usepackage[hyperindex,bookmarksnumbered,plainpages]{hyperref}

\newcommand{\dcl}{\DeclareMathOperator}
\dcl\cdet{cdet} \dcl\Sp{Specm} \dcl\depth{depth} \dcl\im{Im} \dcl\Span{span} \dcl\Ker{Ker} \dcl\Specm{Specm}
\dcl\Supp{Supp} \dcl\codim{codim} \dcl\Y{Y} \dcl\gl{\mathfrak{gl}}    \dcl\U{U} \dcl\T{T}
\dcl\qdet{qdet} \dcl\sgn{sgn} \dcl\gr{gr} \dcl\diag{diag}
\dcl\g{\mathfrak{g}} \dcl\C{\mathbb C} \dcl\dd{{\mathrm d}}

\newcommand\sm{{\mathsf m}}

\newcommand\Ga{{\Gamma}}

\def\cD{\mathcal D}

\newlength\yStones
\newlength\xStones
\newlength\xxStones

\makeatletter
\def\Stones{\pst@object{Stones}}
\def\Stones@i#1{%
  \pst@killglue%
  \begingroup%
  \use@par%
  \setlength\xxStones{\xStones}%
  \expandafter\Stones@ii#1,,\@nil
  \endgroup
  \global\addtolength\xStones{0.6cm}%
  \global\addtolength\yStones{-7.5mm}}%
\def\Stones@ii#1,#2,#3\@nil{%
  \rput(\xxStones,\yStones){%
    \psframebox[framesep=0]{%
      \parbox[c][6mm][c]{11mm}{\makebox[11mm]{$#1$}}}}%
  \addtolength\xxStones{1.2cm}%
  \ifx\relax#2\relax\else\Stones@ii#2,#3\@nil\fi}
\makeatother

\def\Stone#1{\fbox{\makebox[13mm]{\strut#1}}\kern2pt}

\newtheorem{theorem}{Theorem}[section]
\newtheorem{lemma}[theorem]{Lemma}
\newtheorem{corollary}[theorem]{Corollary}
\newtheorem{proposition}[theorem]{Proposition}

\newtheorem{remark}[theorem]{Remark}

\newtheorem{definition}[theorem]{Definition}

\begin{document}
\title{Families of irreducible singular Gelfand-Tsetlin modules of $\displaystyle \mathfrak{gl}(n)$}

\author{Carlos A. Gomes}
\address{Instituto de Matem\'atica e Estat\'istica, Universidade de S\~ao
Paulo,  S\~ao Paulo SP, Brasil} \email{cgomes@ime.usp.br}
\author{Luis Enrique Ramirez}
\address{Universidade Federal do ABC, Santo Andr\'e SP, Brasil} \email{luis.enrique@ufabc.edu.br,}

\begin{abstract}
We prove a conjecture for the  irreducibility of singular Gelfand-Tsetlin modules announced in \cite{FGR3}. We describe explicitly  the irreducible subquotients  of  certain classes of singular Gelfand-Tsetlin modules.
\end{abstract}

\maketitle

\section{Introduction}

In 1950, I. Gelfand and M. Tsetlin \cite{GT} constructed a  basis for any irreducible finite-dimensional $\gl(n)$-module. These  bases are parameterized by the so-called Gelfand-Tsetlin tableaux, whose entries satisfy some integer relations. Observing that the coefficients in the Gelfand-Tsetlin formulas are rational functions on the entries of the tableaux, it is natural to extend the Gelfand-Tsetlin construction to more general modules. For a generic tableau $T(L)$ (no integer relations between elements of the same row) the corresponding module was constructed in \cite{DFO3} and explicit bases for their irreducible subquotients were given in \cite{FGR} providing explicitly new irreducible modules for  
$\mathfrak{gl}(n)$.  The first attempt of generalization is the case when there is only one row and a unique pair of entries in this row with integral condition ($1$-singularity). This cases was treated in \cite{FGR3}. In this construction besides regular tableaux appear new 
 tableauxwhich are  called derivatives tableaux. The vector space $V(T(\bar{v}))$ generated by regulars and derivatives tableaux has a $\mathfrak{gl}(n)$-module structure. In  \cite{FGR3}, V. Futorny, D. Grantcharov and L.E. Ramirez  gave   sufficient condition for irreducibility of the module  $V(T(\bar{v}))$ and conjectured that this same condition is necessary.  In particular,   for $n=3$ this was shown in \cite{FGR}. \\

\

{\bf Conjecture}: Let $n\geq 2$, $V(T(\bar{v}))$  the $1$-singular $\gl(n)$-module. If $V(T(\bar{v}))$  is irreducible then the differences between elements of neighboring  rows of $T(\bar{v})$ are not integers.\\
\

  In the current paper we give a positive answer to this conjecture and describe the irreducible subquotients for certain families of $1$-singular modules $V(T(\bar{v}))$.
 
The paper is organized as follows. In Section $ 2$ we introduce the necessary definitions and notations used through the paper. In Section $3$ we recall the definition of Gelfand-Tsletlin subalgebras and Gelfand-Tsetlin modules. In the same section we recall the Gelfand-Tsetlin theorem about realizations of irreducible finite dimensional modules via Gelfand-Tsetlin tableaux. Also, we recall the construction of generic  Gelfand-Tsetlin modules  presented in \cite{DFO3}. In Section $ 4$ we  discuss  $1$-singular Gelfand-Tsetlin tableaux and modules  constructed in \cite{FGR3}. In the following section  we  define a preorder relation in the set of all tableaux (regular end derivatives) and 
establish  important properties of this relation that will be used  in the next section. Finally, in Ssections $6$ and $7$ we establish  main results in this paper; the first result gives us an explicit basis for an irreducible subquotient of the $1$-singular module $V(T(\bar{v}))$ that contains a given tableau, and the second result gives a positive answer for the Conjecture. \\

\

\noindent{\bf Acknowledgements.} We would like to thank Vyacheslav Futorny for stimulating discussions during the preparation of this paper.

\section{Conventions and notation} The ground field will be ${\mathbb C}$.  For $a \in {\mathbb Z}$, we write $\mathbb Z_{\geq a}$ for the set of all integers $m$ such that $m \geq a$.  We fix an integer $n\geq 2$. By $\gl(n)$ we denote the general linear Lie
algebra consisting of all $n\times n$ complex matrices, and by $\{E_{i,j}\mid 1\leq i,j \leq n\}$  - the
standard basis of $\gl(n)$ of elementary matrices. We fix the standard triangular decomposition and  the corresponding basis of simple roots of $\gl(n)$.  The weights of $\gl(n)$ will be written as $n$-tuples $(\lambda_1,...,\lambda_n).$ 

For a Lie algebra ${\mathfrak a}$ by $U(\mathfrak a)$ we denote the universal enveloping algebra of ${\mathfrak a}$. Throughout this paper $U = U(\gl (n))$.  For a commutative ring $R$, by ${\rm Specm}\, R$ we denote the set of maximal ideals of $R$.

We will write the vectors in $\mathbb{C}^{\frac{n(n+1)}{2}}$ in the following form:
$$
v=(v_{n1},...,v_{nn}|v_{n-1,1},...,v_{n-1,n-1}| \cdots|v_{21}, v_{22}|v_{11}).
$$
For $1\leq j \leq i \leq n$, $\delta^{ij} \in {\mathbb Z}^{\frac{n(n+1)}{2}}$ is defined by  $(\delta^{ij})_{ij}=1$ and all other $(\delta^{ij})_{m\ell}$ are zero. For $i>0$ by $S_i$ we denote the $i$th symmetric group. By $(m, \ell)$ we denote the transposition of $S_i$ switching $m$ and $\ell$.  

\section{Gelfand-Tsetlin modules}

\subsection{Definitions}
Recall that $U=U(\gl (n) )$.  Let  for $m\leqslant n$, let $\mathfrak{gl}_{m}$ be the Lie subalgebra
of $\gl (n)$ spanned by $\{ E_{ij}\,|\, i,j=1,\ldots,m \}$. We have the following chain
$$\gl_1\subset \gl_2\subset \ldots \subset \gl_n,$$
which induces  the chain $U_1\subset$ $U_2\subset$ $\ldots$ $\subset
U_n$ of the universal enveloping algebras  $U_{m}=U(\gl_{m})$, $1\leq m\leq n$. Let
$Z_{m}$ be the center of $U_{m}$. Then $Z_m$ is the polynomial
algebra in the $m$ variables $\{ c_{mk}\,|\,k=1,\ldots,m \}$,
\begin{equation}\label{equ_3}
c_{mk } \ = \ \displaystyle {\sum_{(i_1,\ldots,i_k)\in \{
1,\ldots,m \}^k}} E_{i_1 i_2}E_{i_2 i_3}\ldots E_{i_k i_1}.
\end{equation}

Following \cite{DFO3}, we call the subalgebra of $U$ generated by $\{
Z_m\,|\,m=1,\ldots, n \}$ the \emph{(standard) Gelfand-Tsetlin
subalgebra} of $U$ and will be denoted by  ${\Ga}$. In fact,  ${\Ga}$ is the polynomial algebra in the $\displaystyle \frac{n(n+1)}{2}$ variables $\{
c_{ij}\,|\, 1\leqslant j\leqslant i\leqslant n \}$ (\cite{Zh}).
 Let $\Lambda$ be the polynomial
algebra in the variables $\{\lambda_{ij}\,|$ $1\leqslant j\leqslant
i\leqslant n \}$.

Let $\imath:{\Ga}{\longrightarrow}$ $\Lambda$ be the embedding
defined by $\imath(c_{mk}) \  = \  \gamma_{mk} (\lambda)$, where
\begin{equation} \label{def-gamma}
\gamma_{mk} (\lambda): = \ \sum_{i=1}^m
(\lambda_{mi}+m-1)^k \prod_{j\ne i} \left( 1 -
\frac{1}{\lambda_{mi}-\lambda_{mj}} \right).
\end{equation} The image of $\imath$
coincides with the subalgebra of $G$-invariant polynomials  in
$\Lambda$, where $G:=S_{n}\times\cdots\times S_{1}$ (\cite{Zh}) which we identify with $\Ga$.

\begin{definition}
\label{definition-of-GZ-modules} A finitely generated $U$-module
$M$ is called a \emph{Gelfand-Tsetlin module (with respect to
$\Ga$)} if $M$ splits into  a direct sum
of $\Ga$-modules:

\begin{equation*}
M=\bigoplus_{\sm\in\Sp\Ga}M(\sm),
\end{equation*}
where $$M(\sm)=\{v\in M\ |\ \sm^{k}v=0 \text{ for some }k\geq 0\}.$$
\end{definition}

Identifying $\sm$ with the homomorphism $\chi:\Gamma \rightarrow {\mathbb C}$ with $\Ker \chi=\sm$, we will call $\sm$ 
a \emph{Gelfand-Tsetlin character} of $M$ if $ M(\sm) \neq 0$, and $\dim M(\sm)$ - the \emph{Gelfand-Tsetlin multiplicity of $\sm$}. The  \emph{Gelfand-Tsetlin support} of a Gelfand-Tsetlin module $M$ is the set of all Gelfand-Tsetlin characters of $M$.
\begin{remark}
Note that any irreducible Gelfand-Tsetlin  module over $\gl(n)$ is a weight module with respect to the
 standard Cartan subalgebra $\mathfrak h$ spanned by $E_{ii}$, $i=1,\ldots, n$. In particular, every highest weight module or, more generally, every module from the category $\mathcal O$ is a Gelfand-Tsetlin module.
 \end{remark}

\subsection{Finite dimensional modules for $\mathfrak{gl}(n)$}

In this section we recall a classical result of I. Gelfand and M. Tsetlin which provides an explicit basis  for every irreducible finite dimensional $\mathfrak{gl}(n)$-module.

\begin{definition} The following array  $T(v)$ of complex  numbers $\{v_{ij}\ | \ 1\leq j\leq i\leq n\}$ 
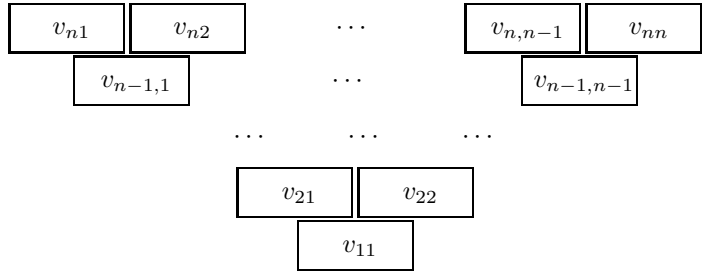
\medskip
\begin{center}

\Stone{\mbox{ $v_{n1}$}}\Stone{\mbox{ $v_{n2}$}}\hspace{1cm} $\cdots$ \hspace{1cm} \Stone{\mbox{ $v_{n,n-1}$}}\Stone{\mbox{ $v_{nn}$}}\\[0.2pt]
\Stone{\mbox{ $v_{n-1,1}$}}\hspace{1.7cm} $\cdots$ \hspace{1.8cm} \Stone{\mbox{ $v_{n-1,n-1}$}}\\[0.3cm]
\hspace{0.2cm}$\cdots$ \hspace{0.8cm} $\cdots$ \hspace{0.8cm} $\cdots$\\[0.3cm]
\Stone{\mbox{ $v_{21}$}}\Stone{\mbox{ $v_{22}$}}\\[0.2pt]
\Stone{\mbox{ $v_{11}$}}\\
\medskip
\end{center}
is called a \emph{Gelfand-Tsetlin tableau}. 
\end{definition}

A Gelfand-Tsetlin tableau  $T(v)$  is called \emph{standard} if:
$$v_{ki}-v_{k-1,i}\in\mathbb{Z}_{\geq 0} \hspace{0.3cm} and \hspace{0.3cm} v_{k-1,i}-v_{k,i+1}\in\mathbb{Z}_{> 0}, \hspace{0.2cm}\text{ for all } 1\leq i\leq k\leq n.$$

\begin{theorem}[\cite{GT}]\label{Gelfand-Tsetlin theorem}
Let $L(\lambda)$ be the finite dimensional irreducible module over $\mathfrak{gl}(n)$ of highest weight $\lambda=(\lambda_{1},\ldots,\lambda_{n})$. Then there exists a basis of $L(\lambda)$ consisting of all standard tableaux $T(v)$ with fixed top row $v_{n1}=\lambda_1,v_{n2}=\lambda_2-1,\ldots,v_{nn}=\lambda_{n}-n+1$. Moreover,  the action of the generators of $\mathfrak{gl}(n)$ on $L(\lambda)$ is given by the  \emph{Gelfand-Tsetlin formulas}:
\begin{equation}\label{GT formulas}
\begin{split}E_{k,k+1}(T(v))&=-\sum_{i=1}^{k}\left(\frac{\prod_{j=1}^{k+1}(v_{ki}-v_{k+1,j})}{\prod_{j\neq i}^{k}(v_{ki}-v_{kj})}\right)T(v+\delta^{ki}),\\
E_{k+1,k}(T(v))&=\sum_{i=1}^{k}\left(\frac{\prod_{j=1}^{k-1}(v_{ki}-v_{k-1,j})}{\prod_{j\neq i}^{k}(v_{ki}-v_{kj})}\right)T(v-\delta^{ki}),\\
E_{kk}(T(v))&=\left(k-1+\sum_{i=1}^{k}v_{ki}-\sum_{i=1}^{k-1}v_{k-1,i}\right)T(v).
\end{split}
\end{equation}
If the new tableau $T(v\pm\delta^{ki})$ is not standard, then the corresponding summand of $E_{k,k+1}(T(v))$ or $E_{k+1,k}(T(v))$ is zero by definition.  Furthermore, the action of generators $c_{rs}$ of $\Gamma$ defined by  (\ref{equ_3}) is given by, 
\begin{equation}\label{action of Gamma in finite dimensional modules}
c_{rs}(T(v))=\gamma_{rs}(v)T(v),
\end{equation}
where $\gamma_{rs}$ are defined in (\ref{def-gamma}).
\end{theorem}

\subsection{Generic Gelfand-Tsetlin modules}\label{section generic modules}
In the case when all denominators are nonintegers, one can use the same formulas and define a new class of infinite dimensional $\gl(n)$-modules:  {\it generic}  Gelfand-Tsetlin modules (cf. \cite{DFO3}, Section 2.3).
\begin{definition}\label{generic tableau definition}
A Gelfand-Tsetlin tableau $T(v)$ is called \emph{generic} if $v_{rs}-v_{ru}\notin\mathbb{Z}$ for each $1\leq s<u\leq r\leq n-1$. 
\end{definition}
\begin{theorem}[\S2.3 in \cite{DFO3} and Theorem 2 in \cite{Maz2}]\label{Generic Gelfand-Tsetlin modules}
Let $T(v)$ be a generic  Gelfand-Tsetlin tableau. Denote by $V(T(v))$  the vector space with basis  consisting of all Gelfand-Tsetlin tableaux $T(w)$ satisfying $w_{nj}=v_{nj}$, $w_{ij}-v_{ij}\in\mathbb{Z}$ for , $1\leq j\leq i \leq n-1$. 
\begin{itemize}
\item[(i)] The vector space $V(T(v))$ has  a structure of a $\mathfrak{gl}(n)$-module with action of the generators of $\mathfrak{gl}(n)$ given by the Gelfand-Tsetlin formulas (\ref{GT formulas}). The module $V(T(v))$ has finite length.
\item[(ii)] The action of the generators of $\Gamma$ on the basis elements of $V(T(v))$ is given by (\ref{action of Gamma in finite dimensional modules}). 
\item[(iii)] The module defined  in {\rm (i)} is a Gelfand-Tsetlin module all of whose Gelfand-Tsetlin multiplicities are $1$.
\item[(iv)] The action of the generators of  $\Gamma$ is given by (\ref{action of Gamma in finite dimensional modules}).
\end{itemize}
\end{theorem}
We will  denote  the module constructed in  Theorem \ref{Generic Gelfand-Tsetlin modules} by $V(T(v))$. Note that $V(T(v))$ need not to be irreducible.  Because $\Gamma$ has simple spectrum on $V(T(v))$ for $T(w)$ in $V(T(v))$ we may define the \emph{irreducible $\mathfrak{gl}(n)$-module in $V(T(v))$ containing $T(w)$} to be the subquotient of $V(T(v))$ containing $T(w)$ (see Theorem \ref{Generic Gelfand-Tsetlin modules}(i)).
A basis for the irreducible subquotients of $V(T(v))$  can be described in terms of the following set.
\begin{equation}\label{definition of Omega+}
\Omega^{+}(T(w)):=\{(r,s,u)\ |\ w_{rs}-w_{r-1,u}\in \mathbb{Z}_{\geq 0}\}.
\end{equation}
\begin{theorem}[Theorem 6.14, in \cite{FGR2}]\label{Basis for irreducible generic modules gl(n)}
Let $T(v)$ be a generic tableau and let $T(w)$ be a tableau in $V(T(v))$. Then the following hold.
\begin{itemize}
\item[(i)] The submodule of $V(T(v))$ generated by $T(w)$  has basis $$\mathcal{N}(T(w)):=\{T(w')\in V(T(v))\ |\ \Omega^{+}(T(w))\subseteq \Omega^{+}(T(w'))\};$$
\item[(ii)] The irreducible $\mathfrak{gl}(n)$-module in $V(T(v))$ containing $T(w)$ has basis
$$\mathcal{I}(T(w)):=\{T(w')\in V(T(v))\ |\ \Omega^{+}(T(w))=\Omega^{+}(T(w'))\}.$$
\end{itemize}
The action of $\mathfrak{gl}(n)$ on both $\mathcal{N}(T(w))$ and $\mathcal{I}(T(w))$ is given by the Gelfand-Tsetlin formulas. 
\end{theorem}

\subsection{Gelfand-Tsetlin formulas in terms of permutations}
\medskip

In this subsection we rewrite and generalize the Gelfand-Tsetlin formulas in Theorem \ref{Gelfand-Tsetlin theorem} in convenient for us terms.  

Recall the convention  that for a vector  $v=(v_{n1},...,v_{nn}|\cdots|v_{11})$ in $\mathbb{C}^{\frac{n(n+1)}{2}}$, by $T(v)$, we  denote the corresponding to $v$ Gelfand-Tsetlin tableau.  Let us call $v$ in ${\mathbb C}^{\frac{n(n+1)}{2}}$ {\it generic} if $T(v)$ is a generic Gelfand-Tsetlin tableau, and denote by ${\mathbb C}^{\frac{n(n+1)}{2}}_{\rm gen}$ the set of all generic vectors in ${\mathbb C}^{\frac{n(n+1)}{2}}$.

Remember that $G=S_n \times \cdots \times  S_1$. Let $\widetilde{S}_m$ denotes the subset of $S_m$ consisting of the transpositions $(1,i)$, $i=1,...,m$. Also, we consider every $\sigma\in \widetilde{S}_m $ as an element of $G$ by letting  $\sigma [t] = \mbox{Id}$ whenever  $t\neq m$.

\begin{definition} \label{def-e-lm}
Let $1 \leq r <\leq n-1$. Set
$$
\varepsilon_{r,r+1}:=\delta^{r,1}\in \mathbb{Z}^{\frac{n(n+1)}{2}}.
$$
Furthermore, define $\varepsilon_{rr}=0$ and  $\varepsilon_{r+1,r}=- \delta^{r,1}$.
\end{definition}

\begin{definition}\label{definition of coefficients e_rs}
For each generic vector $w$ and any $1\leq r, s\leq n$ we define
\begin{equation}\label{definition of e_ij}
\begin{split}
e_{r,r+1}(w)&:=- \frac{\prod_{j=1}^{r+1}(w_{r1}-w_{r+1,j})}{\prod_{j\neq 1}^{r}(w_{r1}-w_{rj})},\\
e_{r+1,r}(w)&:= \frac{\prod_{j=1}^{r-1}(w_{r1}-w_{k-1,j})}{\prod_{j\neq 1}^{k}(w_{r1}-w_{rj})},\\
e_{rr}(w)&:=r-1+\sum\limits_{i=1}^{r}w_{ri}-\sum\limits_{i=1}^{r-1}w_{r-1,i}
\end{split}
\end{equation}
\end{definition}

By using permutations we can rewrite the Gelfand-Tsetlin formulas (\ref{GT formulas}) as follows:

\begin{proposition}\label{GT formulas in terms of permutations} Let $v \in {\mathbb C_{\rm gen}^{\frac{n(n+1)}{2}}}$. The Gelfand-Tsetlin formulas for the generic Gelfand Tsetlin $\mathfrak{gl}(n)$-module  $V(T(v))$ can be written as follows:
$$
E_{\ell m} (T(v+z))= \sum_{\sigma \in \tilde{S}_{r}} e_{\ell m} (\sigma (v+z)) T(v+z+\sigma(\varepsilon_{\ell m})),$$
where $(\ell,m)\in\{(r,r+1), (r+1,r), (r,r)\}$ and $z \in {\mathbb Z^{\frac{n(n+1)}{2}}}$ has top row zero. 
\end{proposition}

\section{Singular Gelfand-Tsetlin modules}\label{sec-der}
In this section we will remember the construction of singular Gelfand-Tsetlin modules given in \cite{FGR3}.

\begin{definition}
A vector $v\in\mathbb{C}^{\frac{n(n+1)}{2}}$ will be called \emph{singular} if there exist $1\leq s<t\leq r\leq n-1$ such that $v_{rs}-v_{rt}\in\mathbb{Z}$. The vector $v$ will be called  \emph{$1$-singular} if  there exist $i,j,k$ with $1\leq i < j \leq k \leq n-1$ such that $v_{ki}-v_{kj}\in\mathbb{Z}$ and $v_{rs}-v_{rt}\notin\mathbb{Z}$ for all $(r,s,t) \neq (k,i,j)$. 
\end{definition}

{\it From now on we fix $(i,j,k)$ such that  $1\leq i < j \leq k\leq n-1$}. In \cite{FGR3}, associated with any $1$-singular tableau $T(\bar{v})$ is constructed a Gelfand-Tsetlin module $V(T(\bar{v}))$ and explicit formulas for the action of the generators of $\mathfrak{gl}(n)$ and the generators of $\Gamma$ is given, in this section we will remember  that construction.

Let us fix a $1$-singular vector $\bar{v}$ such that $\bar{v}_{ki}-\bar{v}_{kj}=0$. {\it From now on by $\tau$ we denote  the element $(\tau[n],\ldots,\tau[2],\tau[1])$ in $G$ such that $\tau[k]$ is the transposition $(i,j)$  and all other $\tau[t]$ are $\mbox{Id}$.}  We formally introduce  new tableaux  ${\mathcal D} T({\bar{v}} + w)$ for  every $w \in {\mathbb Z}^{\frac{n(n-1)}{2}}$ subject to the relations ${\mathcal D} T({\bar{v}} + w) + {\mathcal D} T({\bar{v}} + \tau(w)) = 0$. We call ${\mathcal D} T(\bar{v}+w)$  {\it the derivative Gelfand-Tsetlin tableau} associated with $w$. 

\begin{definition}
We set  $V(T(\bar{v}))$ to be the vector space spanned by the set of tableaux $\{ T(\bar{v} + w), \,\mathcal{D} T(\bar{v} + w) \; | \; w \in {\mathbb Z}^{\frac{n(n-1)}{2}}\}$, subject to the relations $T(\bar{v} + w) = T(\bar{v} +\tau(w))$ and ${\mathcal D} T(\bar{v} + w) + {\mathcal D} T(\bar{v} +\tau(w)) = 0$. We also fix a basis of $V(T(\bar{v}))$ to be the  set $\{Tab(w):w\in\mathbb{Z}^{\frac{n(n-1)}{2}}\}$, where 
$$Tab(w):=\begin{cases}
T(\bar{v}+w), & \mbox{if } \ w_{ki}-w_{kj}\leq 0\\
\mathcal{D} T(\bar{v}+w), & \mbox{if } \ w_{ki}-w_{kj}> 0
\end{cases}$$
\end{definition}

For a variables vector $v$ and $f$ is a rational function on variables $v_{rs}$ which is smooth on the hyperplane $v_{ki}-v_{kj}=0$ we can define  the linear map 
$$
\mathcal{D}^{\bar{v}} (f T(v+z)) = \mathcal{D}^{\bar{v}} (f) T(\bar{v}+z) +   f(\bar{v}) \mathcal{D} T(\bar{v}+z),
$$ 
where $\mathcal{D}^{\bar{v}}(f) = \frac{1}{2}\left(\frac{\partial f}{\partial v_{ki}}-\frac{\partial f}{\partial v_{kj}}\right)(\bar{v})$.
The following lemma will be useful in order to do some computations.
\begin{lemma}\label{Properties of D on functions}
Let $f$ be a rational function on variables $v_{rs}$ smooth on the hyperplane $v_{ki}-v_{kj}=0$. Then,
\begin{itemize}
\item[(i)] $\mathcal{D}^{\bar{v}}((v_{ki}-v_{kj})f)=f(\bar{v})$.
\item[(ii)] If $f$ is symmetric with respect to $v_{ki}$ and $v_{kj}$ then, $\mathcal{D}^{\bar{v}}(f)=0$.
\end{itemize}
\end{lemma}

\begin{theorem}[\cite{FGR3} Theorems 4.11 and 4.12]
$ V(T(\bar{v}))$ is an $1$-singular  Gelfand-Tsetlin $\mathfrak{gl} (n)$-module, with action of the generators of $\mathfrak{gl}(n)$ given by
\begin{equation}\label{GT formulas for regular tableaux}
E_{rs}(T(\bar{v} + z))=\  \mathcal{D}^{\bar{v}}((v_{ki} - v_{kj})E_{rs}(T(v + z)))
\end{equation}
\begin{equation}\label{GT formulas for derivative tableaux}
E_{rs}(\mathcal{D}T(\bar{v} + z')))=\ \mathcal{D}^{\bar{v}} ( E_{rs}(T(v + z'))),
\end{equation}
and the action of the generators of $\Gamma$ can be written explicitly as follows:
\begin{equation}\label{Gamma acting on T}
c_{rs}(T(\bar{v}+z))=\gamma_{rs}(\bar{v}+z)T(\bar{v}+z)
\end{equation}
\begin{equation}\label{Gamma acting on DT}
c_{rs}(\mathcal{D}T(\bar{v}+z'))= \gamma_{rs}(\bar{v}+z')\mathcal{D}T(\bar{v}+z')+\mathcal{D}^{\overline{v}}(\gamma_{rs}(v+z'))T(\bar{v}+z')
\end{equation}
for any $z,z'\in\mathbb{Z}^{\frac{n(n-1)}{2}}$ with $z'\neq\tau(z')$.
\end{theorem}

For some of the generators of $\mathfrak{gl}(n)$ the action on $V(T(\bar{v}))$ coincide with the classical Gelfand-Tsetlin formulas. The following corollary gives some sufficient conditions in order to have this property. 

\begin{corollary}\label{elements of gl(n) acting by the classical GT formulas}
Let $z$ be any element of $\mathbb{Z}^{\frac{n(n-1)}{2}}$.
\begin{itemize}
\item[(i)] For any $\ell ,m$ such that $k<\min\{\ell, m\}$ or $\max\{\ell, m\}\leq k$ we have
$$
E_{\ell m} (T(\bar{v}+z))= \sum_{\sigma \in \Phi_{\ell m}} e_{\ell m} (\sigma (\bar{v}+z)) T(\bar{v}+z+\sigma(\varepsilon_{\ell m})).$$
\item[(ii)] The equality 
$$
E_{rs} (\mathcal{D}T(\bar{v}+z))= \sum_{\sigma \in \Phi_{rs}} e_{rs} (\sigma (\bar{v}+z)) \mathcal{D}T(\bar{v}+z+\sigma(\varepsilon_{rs})),$$
holds whenever $r,s$ satisfy:
$$
\begin{cases}
k\notin\{r,r+1,\ldots,s\}, & \text{if }\ \  r<s,\\
k\notin\{s,s+1,\ldots,r\}, & \text{if }\ \  s<r.
\end{cases}
$$ 
\end{itemize}
\end{corollary}

\begin{proof} The proof will be based on the basic properties of the rational functions $e_{ij}(w)$ defined in \ref{definition of coefficients e_rs}. 
\begin{itemize} 
\item[(i)] If $k<\min\{\ell, m\}$ or $\max\{\ell, m\}-1<k$ then, $e_{\ell m}(\sigma(v+z))$ is a smooth function, for any $\sigma \in \Phi_{\ell m}$ so, by Lemma \ref{Properties of D on functions} we have $\mathcal{D}^{\bar{v}}((v_{ki}-v_{kj})e_{\ell m}(\sigma(v+z)))=e_{\ell m}(\sigma(\bar{v}+z))$ and $ev(\bar{v})((v_{ki}-v_{kj})e_{\ell m}(\sigma(v+z)))=0$. Now, by (\ref{GT formulas for regular tableaux}) we have the equality.
\item[(ii)] Under this conditions on $r,s$ the function $e_{rs}(\sigma(v+z))$ is symmetric with respect to $v_{ki}$ and $v_{kj}$ for any  $\sigma \in \Phi_{rs}$. Therefore, by Lemma \ref{Properties of D on functions}, $\mathcal{D}^{\bar{v}}e_{rs}(\sigma(v+z)))=0$ and then, by (\ref{GT formulas for derivative tableaux}) we have the desired equality.
\end{itemize}
\end{proof}

\begin{remark}\label{GT subspaces}
Note that $\chi\in\Supp_{GT}(V(T(\bar{v})))$ if and only if there exist $w\in\mathbb{Z}^{\frac{n(n-1)}{2}}$ such that $\chi(c_{rs})=\gamma_{rs}(\bar{v}+w)$ for any $1\leq s\leq r\leq n$. In this situation each Gelfand-Tsetlin subspace $V(T(\bar{v}))(\chi)$ is generated by $T(\bar{v}+w)$ and $\mathcal{D}T(\bar{v}+w)$. 
\end{remark}

\begin{remark}\label{remark all gamma acts same}
As the polynomials $\{\gamma_{rs}(v)\}_{1\leq s\leq r\leq n}$ are symmetric in the entries of $v$ and generated all symmetric polynomials, we have $\gamma_{rs}(v)=\gamma_{rs}(v')$  for any $r,s$ if only if, $v=\sigma(v')$ for some $\sigma \in G$.  In particular, for the $1$-singular vector $\bar{v}$ we have $\gamma_{rs}(\bar{v}+z)=\gamma_{rs}(\bar{v}+w)$ for any $1\leq s\leq r\leq n$ if, and only if,  $w=z$ or $w=\tau(z)$.
\end{remark}

\begin{lemma}[\cite{FGR3} Lemma 5.2]\label{lem-ck2}
Assume $\tau(z)\neq z$. Then:
\begin{itemize}
\item[(i)] $(c_{k2}-\gamma_{k2}(\bar{v}+z))T(\bar{v}+z)=0$.
\item[(ii)] $(c_{k2}-\gamma_{k2}(\bar{v}+z))\mathcal{D}T(\bar{v}+z)=\mathcal{D}^{\bar{v}}(\gamma_{k2}(v+z))T(\bar{v}+z)$ with $\mathcal{D}^{\bar{v}}(\gamma_{k2}(v+z))\neq 0$.
\item[(iii)] $(c_{k2}-\gamma_{k2}(\bar{v}+z))^{2}\mathcal{D}T(\bar{v}+z)=0.$
\end{itemize}
\end{lemma}


\section{$\Gamma$ separates tableaux in $V(T(\bar{v}))$}

One essential property of generic Gelfan-Tsetlin modules  described in Theorem \ref{Generic Gelfand-Tsetlin modules} is that for any two different tableaux in $V(T(\bar{v}))$ there exists an element $\gamma$ of $\Gamma$
 that separates those tableaux (i.e. the action of $\gamma$
 has different eigenvalues on this two tableaux). In this section we will give a detailed prove of this fact for any $1$-singular module $V(T(\bar{v}))$. Remember that we fix a basis $\{Tab(z)\ |\ z\in\mathbb{Z}^{\frac{n(n-1)}{2}}\}$ for $V(T(\bar{v}))$.\\
For any $1\leq s\leq r\leq n$ and $z\in \mathbb{Z}^{\frac{n(n-1)}{2}}$ we will denote by $C_{rs}(z)$ the element $c_{rs}-\gamma_{rs}(\bar{v}+z)$ of $\Gamma$.
\begin{lemma}\label{C_rs acts on Gelfand-Tsetlin subspaces}
For any $1\leq s\leq r\leq n$, any $z\in\mathbb{Z}^{\frac{n(n-1)}{2}}$  and $\chi\in Supp_{GT}(V(T(\bar{v})))$, the subspace $V(T(\bar{v}))(\chi)$ is $C_{rs}(z)$-invariant.
\end{lemma}
\begin{proof}
By Remark \ref{GT subspaces} any Gelfand-Tsetlin subspace $V(T(\bar{v}))(\chi)$ is generated by $T(\bar{v}+w)$ and $\mathcal{D}T(\bar{v}+w)$ for some $w\in\mathbb{Z}^{\frac{n(n-1)}{2}}$. Now,
\begin{align*}
C_{rs}(z)T(\bar{v}+w)=&(\gamma_{rs}(\bar{v}+w)-\gamma_{rs}(\bar{v}+z))T(\bar{v}+w)\\
C_{rs}(z)\mathcal{D}T(\bar{v}+w)=&\mathcal{D}^{\bar{v}}(\gamma_{rs}(\bar{v}+z))T(\bar{v}+w)+(\gamma_{rs}(\bar{v}+w)-\gamma_{rs}(\bar{v}+z))\mathcal{D}T(\bar{v}+w).
\end{align*}
\end{proof}
\begin{definition}
Given $z,w \in \mathbb{Z}^{\frac{n(n-1)}{2}}$, we write $Tab(z) \prec Tab(w)$ if, and only if, there exists $u \in U(\mathfrak{gl}(n))$ such that $Tab(w)$ appears with nonzero coefficient in the decomposition of  $u\cdot Tab(z)$ as a linear combination of tableaux. 
\end{definition}

\begin{lemma}\label{Derivative generates normal tableau}
Let $w \in \mathbb{Z}^{\frac{n(n+1)}{2}}$ be such that $w \neq \tau(w)$, then $\mathcal{D}T(\bar{v}+w)\prec T(\bar{v}+w)$.
\end{lemma}

\begin{proof}
By Lemma \ref{lem-ck2}(ii) we have, $C_{k2}(w)\cD T(\bar{v}+w)=\cD^{\bar{v}}(\gamma_{k2}(v+w))T(\bar{v}+w)$. As $w \neq \tau(w)$, $\cD^{\bar{v}}(\gamma_{k2}(v+w))\neq 0$ and then, $\cD T(\bar{v}+w) \prec T(\bar{v}+w)$.
\end{proof}

\begin{lemma}\label{separation different characters}
Let $z,w \in \mathbb{Z}^{\frac{n(n-1)}{2}}$ be such that $w \neq z$ and $w \neq \tau(z)$. There exists $\gamma_{z}^{w} \in \Gamma$ such that  
$$\gamma_{z}^{w}\cdot T(\bar{v}+z)=\gamma_{z}^{w}\cdot \cD T(\bar{v}+z)=0  \   \mbox{and}    \  \gamma_{z}^{w}\cdot Tab(w)=Tab(w)$$
\end{lemma}

\begin{proof}
Let us fix $r$ and $s$ such that $\gamma_{rs}(\bar{v}+w)\neq \gamma_{rs}(\bar{v}+z)$ (such $r$, $s$ exist because of Remark \ref{remark all gamma acts same}). Set $a:=\gamma_{rs}(\bar{v}+w)-\gamma_{rs}(\bar{v}+z)\neq 0$, by a direct computation we have the following identities:
\begin{itemize}
\item[(i)] $C_{rs}(z)T(\bar{v}+z)=0=C_{rs}^{2}(z)\cD T(\bar{v}+z$.
\item[(ii)] $C_{rs}(z)\cD T(\bar{v}+z)=\cD^{\bar{v}}(\gamma_{rs}(v+z))T(\bar{v}+z)$. 
\item[(iii)] $C_{rs}(z)T(\bar{v}+w)=aT(\bar{v}+w)$
\item[(iv)] $C_{rs}(z)\cD T(\bar{v}+w)=\cD^{\bar{v}}(\gamma_{rs}(v+w))T(\bar{v}+w)+a\cD T(\bar{v}+w)$.
\item[(v)] $C_{rs}^{2}(z)T(\bar{v}+w)=a^2 T(\bar{v}+w)$
\item[(vi)]$C_{rs}^2(z)\cD T(\bar{v}+w)
=2a\cD^{\bar{v}}(\gamma_{rs}(v+w))T(\bar{v}+w)+a^{2}\cD T(\bar{v}+w)$.
\end{itemize}
Now, we have two cases:

\noindent
{\it Case 1.} Suppose $Tab(w)=T(\bar{v}+w)$. In this case we take $\gamma_{z}^{w}=\frac{1}{a^{2}}C_{rs}^{2}(z)$.

\noindent
{\it Case 2.} Suppose that $Tab(w)=\mathcal{D}T(\bar{v}+w)$.
 In this case we have two possibilities, namely:

\begin{itemize}
\item[(i)] $\cD^{\bar{v}}(\gamma_{rs}(v+w))=0$. In this case, from $(v)$ we have $C_{rs}^{2}(z)\cdot \mathcal{D}T(\bar{v}+w)=a^2\mathcal{D}T(\bar{v}+w)$. So, we can take $\gamma_{z}^{w}=\frac{1}{a^{2}}C_{rs}^{2}(z)$.
\item[(ii)] $\cD^{\bar{v}}(\gamma_{rs}(v+w))\neq 0$. By Lemma \ref{lem-ck2}(ii), $\cD^{\bar{v}}(\gamma_{k2}(v+w))\neq 0$ and 
$$\begin{cases}
C_{k2}(w)T(\bar{v}+w)&=\ 0,\\
C_{k2}(w)\cD T(\bar{v}+w)&=\ \cD^{\bar{v}}(\gamma_{k2}(v+w))T(\bar{v}+w).
\end{cases}$$
So, applying $C_{k2}(w)$ to the equality $(vi)$, we have:
\begin{equation}\label{eq with c_k2}
C_{k2}(w)C_{rs}^{2}(z)\cD T(\bar{v}+w)=a^{2}\cD^{\bar{v}}(\gamma_{k2}(v+w))T(\bar{v}+w)
\end{equation}
Now, replacing Equality (\ref{eq with c_k2}) in $(vi)$, we have: 
\begin{equation}\label{eq with c_rs and c_k2}
\left(1-\frac{2C_{k2}(w)}{a\cD^{\bar{v}}(\gamma_{k2}(v+w))}\right)C_{rs}^{2}(z)\cD T(\bar{v}+w)=a^{2}\cD T(\bar{v}+w).
\end{equation}

Therefore, in this case we can consider $\gamma_{z}^{w}=\frac{1}{a^2}\left(1-\frac{2C_{k2}(w)}{a\cD^{\bar{v}}(\gamma_{k2}(v+w))}\right)C_{rs}^{2}(z)$.
\end{itemize}
Summarizing, we have:

$$
\gamma_{z}^{w}=\begin{cases}
\frac{1}{a^{2}}C_{rs}^{2}(z),& \text{ if }\ \  Tab(w)=T(\bar{v}+w)\ \text{ or }\\&\ \ \ \ \ \ \cD^{\bar{v}}(\gamma_{rs}(v+w))=0\\
\frac{1}{a^2}\left(1-\frac{2C_{k2}(w)}{a\cD^{\bar{v}}(\gamma_{k2}(v+w))}\right)C_{rs}^{2}(z),& \text{ if } \cD^{\bar{v}}(\gamma_{rs}(v+w))\neq 0
 \end{cases} 
 $$
\end{proof}
 \begin{remark}\label{GT subspaces are gamma_zw-invariant} 
Note that each $\gamma_{z}^{w}$ on the previous lemma is a combination of pro-\\ducts of elements of $\Gamma$ of the form $C_{\ell m}(w')$. So, by Lemma \ref{C_rs acts on Gelfand-Tsetlin subspaces}, the Gelfand-Tsetlin subspaces $V(T(\bar{v}))(\chi)$ are $\gamma_{z}^{w}$-invariant.
\end{remark}

\begin{lemma}\label{c_k2 separates}
Let us consider $T:=aT(\bar{v}+w)+ b\cD T(\bar{v}+\tau(w))$, where $w \in \mathbb{Z}^{\frac{n(n-1)}{2}}$ is such that $w\neq \tau(w)$ and $a,b\in\mathbb{C}$.  Then,
\begin{itemize}
\item[(i)] If $a \neq 0$, then $\gamma_1 \cdot T=T(\bar{v}+w)$ for some $\gamma_{1}\in\Gamma$.
\item[(ii)] If $b \neq 0$, then $\gamma_2 \cdot T=\cD T(\bar{v}+\tau(w))$ for some $\gamma_{2}\in\Gamma$.
\end{itemize}
\end{lemma}

\begin{proof} Let $\gamma_{1},\gamma_{2}\in\Gamma$ defined by $$\gamma_1=\begin{cases}\frac{1}{a}, & \text{ if }\ \ b=0\\\frac{C_{rs}^{2}(w)}{b\cD^{\bar{v}}(\gamma_{k2}(v+w))}, & \text{ if }\ \  b\neq 0
\end{cases}\ \ , \ \gamma_2=\begin{cases}\frac{1}{b}, & \text{ if }\ \ a=0\\\frac{1}{b}\left(1-a\frac{C_{rs}^{2}(w)}{a\cD^{\bar{v}}(\gamma_{k2}(v+w))}\right), & \text{ if }\ \  a\neq 0
\end{cases}.$$
First we note that by Lemma \ref{lem-ck2}, the denominators of $\gamma_{1},\gamma_{2}$ are not zero. The rest of the proof is a straightforward verification  
\end{proof}

\begin{theorem}\label{coeff nonzero implies equal}
If $z,w \in \mathbb{Z}^{\frac{n(n-1)}{2}}$ are such that $Tab(z) \prec Tab(w)$ then, there exist $u \in U(\mathfrak{gl}(n))$ such that $u \cdot Tab(z)=Tab(w)$. 
\end{theorem}

\begin{proof}
As $Tab(z) \prec Tab(w)$, there exist $u' \in U(\mathfrak{gl}(n))$ such that $u'\cdot Tab(z)$ can be written as follows
\begin{equation}\label{eq zw}
\sum_{i=0}^{s} a_iT(\bar{v}+w_i)+b_i\cD T(\bar{v}+\tau(w_i))\in V(T(\bar{v}))(\chi_{0})\oplus\cdots \oplus V(T(\bar{v}))(\chi_{s})
\end{equation}
where $w_{i}\neq w_{j}, \tau(w_{j})$ for any $i\neq j$ and $w_{0}=w$ or $w_{0}=\tau(w)$, $a_{0}\neq 0$ or $b_{0}\neq 0$, and $\chi_{i}$ the Gelfand-Tsetlin character associated with $w_{i}$.
By Lemma \ref{separation different characters}, for each $j\in\{1,2\cdots,s\}$, there exist $\gamma_{w_j}^{w_0} \in \Gamma$ such that
$$\gamma_{w_j}^{w_0}T(\bar{v}+w_j)=\gamma_{w_j}^{w_0}\cD T(\bar{v}+w_j)=0 \ \mbox{and} \ \ \gamma_{w_j}^{w_0}Tab(w_0)=Tab(w_0)$$
Then, by Remark \ref{GT subspaces are gamma_zw-invariant}, applying  $\gamma:=\gamma_{w_s}^{w_{0}}\cdots\gamma_{w_1}^{w_{0}}$ to $u'\cdot Tab(z)$, we have
$$
\gamma u'\cdot Tab(z)=\gamma_{w_s}^{w_{0}}\cdots\gamma_{w_1}^{w_{0}}\left(\displaystyle\sum_{i=1}^s a_iT(\bar{v}+w_i)+b_i\cD T(\bar{v}+\tau(w_i))\right)\in V(T(\bar{v}))(\chi_{0})
$$

So, $\gamma u'\cdot Tab(z)=aT(\bar{v}+w)+b\cD T(\bar{v}+\tau(w))$ for some $a,b\in\mathbb{C}$. Let us see the relation between the coefficients $a,\ b$ and $a_{0},\ b_{0}$. 

\noindent
{\it Case 1.} Suppose $Tab(w)=T(\bar{v}+w)$. In this case, for any $j=1,\ldots, s$, we have $\gamma_{w_{j}}^{w}T(\bar{v}+w)=T(\bar{v}+w)$ (by construction of $\gamma_{w_{j}}^{w_{0}}$) and by Remark \ref{GT subspaces are gamma_zw-invariant}, $\gamma_{w_{j}}^{w}\mathcal{D}T(\bar{v}+w)=\mathcal{D}T(\bar{v}+w)$. Therefore, $$\gamma u\cdot Tab(z)=a_{0}T(\bar{v}+w)+b_{0}\cD T(\bar{v}+\tau(w)).$$
Now, as $a_{0}\neq 0$, by Lemma \ref{c_k2 separates} there exist $\gamma_{1}\in\Gamma$ such that $$\gamma_{1}\gamma u\cdot Tab(z)=\gamma_{1}(a_{0}T(\bar{v}+w)+b_{0}\cD T(\bar{v}+\tau(w)))=T(\bar{v}+w).$$

\noindent
{\it Case 2.} Suppose $Tab(w)=\mathcal{D}T(\bar{v}+w)$. In this case, for any $j=1,\ldots, s$, we have $\gamma_{w_{j}}^{w}\mathcal{D}T(\bar{v}+w)=\mathcal{D}T(\bar{v}+w)$ (by construction of $\gamma_{w_{j}}^{w_{0}}$) and  $\gamma_{w_{j}}^{w}T(\bar{v}+w)=\alpha_{j}T(\bar{v}+w)$ for some $\alpha_{j}\in\mathbb{C}$. Therefore, $$\gamma u\cdot Tab(z)=aT(\bar{v}+w)+b_{0}\cD T(\bar{v}+\tau(w)),$$
with $a=\alpha_{s}\cdots\alpha_{1}a_{0}$ and $b_{0}\neq 0$. By Lemma \ref{c_k2 separates} there exist $\gamma_{2}\in\Gamma$ such that $\gamma_{2}\gamma u\cdot Tab(z)=\gamma_{2}(aT(\bar{v}+w)+b_{0}\cD T(\bar{v}+\tau(w)))=\mathcal{D}T(\bar{v}+w)$.
\end{proof}

\begin{corollary}\label{preorder}
The relation ``$\prec$" define a preorder on the set of tableaux $\mathcal{B}(T(\bar{v}))$ (i.e. $\prec$ is a reflexive and transitive relation).
\end{corollary}

\begin{proof}
Reflexivity is clear from the definition of ``$\prec$". For transitivity, assume that $Tab(w_1) \prec Tab(w_2)$ and $Tab(w_2) \prec Tab(w_3)$ for some $w_1,w_2,w_3 \in \mathbb{Z}^{\frac{n(n-1)}{2}}$. By Theorem \ref{coeff nonzero implies equal} there  exists $u_1,u_2 \in U(\mathfrak{gl}(n))$ such that $u_{1}Tab(w_{1})=Tab(w_2)$ and $u_{2}Tab(w_{2})=Tab(w_3)$. Therefore, $u_{2}u_{1}Tab(w_{1})=Tab(w_{3})$. That is, $Tab(w_{1})\prec Tab(w_{3})$
\end{proof}

\section{Irreducible subquotients in $V(T(\bar{v}))$}

The Theorem \ref{Basis for irreducible generic modules gl(n)} provides an explicit basis for an irreducible submodule  that contains a given tableau for generic case. In this section we will present a similar result for $1$-singular case and this will lead us an alternative proof for Theorem 4.14 in \cite{FGR3}.

\begin{definition}
Given $z,w \in \mathbb{Z}^{\frac{n(n-1)}{2}}$, define  the {\bf distance} between the tableaux, $Tab(z)$ and $Tab(w)$  by $$d(z,w)=\sum\limits_{1\leq s\leq r\leq n}|z_{rs}-w_{rs}|.$$
\end{definition}

The Theorem 4.14 in \cite{FGR3} states that if $ \bar{v}_{rs}-\bar{v}_{r-1,t} \notin \mathbb{Z}$, for every $(r,s,t)$, then the module  $V(T(\bar{v}))$ is irreducible. Now consider a tableau such that the condition
$\bar{v}_{rs}-\bar{v}_{r-1,t} \notin \mathbb{Z}$, for every $(r,s,t)$ is satisfied for any $r \geq k+1$. We will show as construct  a basis for a irreducible subquotient of $V(T(\bar{v}))$ that contains a given tableaux. For this we need of some definitions.

\begin{definition} For $w \in \mathbb{Z}^{\frac{n(n-1)}{2}}$, we define
\begin{align*}
\Omega_k(Tab(w))&=\{(r,s,t) \  |  \ r \leq k, (\bar{v}_{rs}+w_{rs})-(\bar{v}_{r-1,t}+w_{r-1,t}) \in \mathbb{Z}\}\\
\Omega_k^+(Tab(w))&=\{(r,s,t) \  |  \ r \leq k, (\bar{v}_{rs}+w_{rs})-(\bar{v}_{r-1,t}+w_{r-1,t}) \in \mathbb{Z}_{\geq 0}\}\\
\mathcal{I}_k(Tab(w))&=\{Tab(w') \in V(T(\bar{v})) \ | \ \Omega_k^+(Tab(w))=\Omega_k^+(Tab(w'))\}
\end{align*}
\end{definition}

\begin{lemma}\label{top part}
Let $Tab(w'), Tab(w^*) \in \mathcal{I}_k(Tab(w))$ be tableaux such that $Tab(w')\nprec Tab(w^*)$, then there exists $i,j$ with $k+1 \leq j\leq i <n$ such that $w'_{ij} \neq w^*_{ij}$.
\end{lemma}
\begin{proof}
The tableaux $Tab(w'), Tab(w^*) \in \mathcal{I}_k(Tab(w))$ can be separated in two parts: the top part, i.e. the part from  $(k+1)$-th row to  $n$-th row and the bottom part from row $1$ to row $k$. Now, suppose that  $w'_{ij}=w^*_{ij}$ for any $k+1 \leq j\leq i <n$. In this case we  have $Tab(w') \prec Tab(w^*)$, because $Tab(w'), Tab(w^*) \in \mathcal{I}_k(Tab(w))$ implies (as in the generic case for $\mathfrak{gl}(k)$) that  bottom parts of $Tab(w')$  and $Tab(w^*)$ (that we denote by $Tab_k(w')$, $Tab_k(w^*)$) are such that $Tab_k(w') \prec Tab_k(w^*)$. 
\end{proof}

\begin{theorem}\label{Theorem: Basis for irreducible modules generalization} 
The set $\mathcal{I}_k(Tab(w))$ is a basis for irreducible  subquotient  of $V(T(\bar{v}))$. 
that contains the tableau $Tab(w)$.
\end{theorem}
\begin{proof}
Let $N=Span_{\mathbb{C}}\mathcal{I}_k(Tab(w))$ be the submodule of $V(T(\bar{v}))$ generated by the set $\mathcal{I}_k(Tab(w))$. If $N$ is not irreducible. By Corollary \ref{preorder}, there exist tableaux $Tab(z')$ and $Tab(w)$ such that $Tab(z')\nprec Tab(w)$. Now, fix $w$ and  choose $z\in\{z'\in\mathbb{Z}^{\frac{n(n+1)}{2}}\ |\  Tab(z')\nprec Tab(w)\}$ such that $d(z,w)$ is minimal. Set $d:=d(z,w)$ then, if $z'\in\mathbb{Z}^{\frac{n(n-1)}{2}}$ is such that $d(z',w)<d$, we should have $Tab(z')\prec Tab(w)$. As $Tab(z)\nprec Tab(w)$, we have $d\geq 1$, that is, $z_{rs}\neq w_{rs}$ for some $r,s$ and by Theorem \ref{top part} we have that is possible for $r \geq k+1$. Now fix the position $r,s$ and assume without lose of generality that $z_{rs}<w_{rs}$. The case $w_{rs}<z_{rs}$ will be analogous.

As $z_{rs}<w_{rs}$, we have $d(z+\delta^{rs},w)=d-1<d$. Then, $Tab(z+\delta^{rs})\prec Tab(w)$. Therefore, as by Corollary \ref{preorder}, $\prec$ is transitive, we should have $Tab(z)\nprec Tab(z+\delta^{rs})$. So, if $Tab(z+\delta^{rs})$ appear in the decomposition of $u\cdot Tab(z)$ for some $u\in U$, the corresponding coefficient should be zero.\\

As the formulas that define the action of $\mathfrak{gl}(n)$ on $V(T(\bar{v}))$ depend of the type of tableau (derivative tableau or regular tableau), we will consider four cases as follows

\begin{center}
\begin{tabular}{|c|c|c|}
\hline
 \ & $Tab(z)$ & $Tab(z+\delta^{rs})$\\ \hline
Case 1\ &$T(\bar{v}+z)$ & $T(\bar{v}+z+\delta^{rs})$\\ \hline
Case 2\ &$\cD T(\bar{v}+z)$ & $\cD T(\bar{v}+z+\delta^{rs})$\\ \hline
Case 3\ &$T(\bar{v}+z)$ & $\cD T(\bar{v}+z+\delta^{rs})$\\ \hline
Case 4\ &$\cD T(\bar{v}+z)$ & $T(\bar{v}+z+\delta^{rs})$\\ \hline
\end{tabular}

\end{center}
{\it Case 1.} The coefficient of $T(\bar{v}+z+\delta^{rs})$ in decomposition $E_{r,r+1}T(\bar{v}+z)$  is given by $\mathcal{D}^{\bar{v}}((v_{ki}-v_{kj})e_{r,r+1}(\sigma_{s}(v+z)))$ where $\sigma_{s}$ is the transposition $(1,s)$ on row $r$ and identity on the other rows. But, $\mathcal{D}^{\bar{v}}((v_{ki}-v_{kj})e_{r,r+1}(\sigma_{s}(v+z)))=e_{r,r+1}(\sigma_{s}(\bar{v}+z))$, because in this case the function $e_{r,r+1}(\sigma_{s}(v+z))$ is smooth. As the numerator of $e_{r,r+1}(\sigma_{s}(\bar{v}+z))$ is a product of differences of type $(\bar{v}_{rs}+z_{rs})-(\bar{v}_{r+1,t}+z_{r+1,t})$. We necessarily have $(\bar{v}_{rs}+z_{rs})-(\bar{v}_{r+1,t}+z_{r+1,t})=0$ for some $t$, that is: $\bar{v}_{r+1,t}-\bar{v}_{r,s} \in \mathbb{Z}$, with $r \geq k+1$, that is a contradiction.

\noindent
{\it Case 2.}  The coefficient of $\mathcal{D}T(\bar{v}+z+\delta^{rs})$ in decomposition $E_{r,r+1}\mathcal{D}T(\bar{v}+z)$ is given by $e_{r,r+1}(\sigma_{s}(\bar{v}+z))$. Analogously the first case, we obtain that $\bar{v}_{r+1,t}-\bar{v}_{r,s} \in \mathbb{Z}$ for some $t$,  with $r \geq k+1$, that is a contradiction.

\noindent
{\it Case 3.} The only possibility in order to have $Tab(z)=T(\bar{v}+z)$ and $Tab(z+\delta^{rs})=\mathcal{D}T(\bar{v}+z+\delta^{rs})$ is $z_{ki}=z_{kj}$ and $(r,s)\in\{(k,i),(k,j)\}$. As the coefficient of $\cD T(\bar{v}+z+\delta^{ks})$ is $ev(\bar{v})((v_{ki}-v_{kj})e_{k,k+1}(\sigma_{s}(v+z))$ and $e_{k,k+1}(\sigma_{s}(v+z))$ has singularity at $v_{ki}-v_{kj}=0$, $ev(\bar{v})((v_{ki}-v_{kj})e_{k,k+1}(\sigma_{s}(v+z))=0$ if, and only if $$\prod_{j=1}^{k+1}((\bar{v}+z)_{ks}-(\bar{v}+z)_{k+1,j})=0$$
which implies that some difference $\bar{v}_{ks}-\bar{v}_{k+1,t} \in \mathbb{Z}$, with $r \geq k+1$, that is a contradiction.

\noindent
{\it Case 4.} In this case we have:
$$\begin{array}{rl}
E_{r,r+1}(\cD T(\bar{v}+z))=&\displaystyle\sum_{\sigma \in \Phi_{r,r+1}}\cD^{\bar{v}} (e_{r,r+1}(\sigma(v+z))T(\bar{v}+z+\sigma(\varepsilon_{r,r+1}))\vspace{0,3cm}\\
+&\displaystyle\sum_{\sigma \in \Phi_{r,r+1}}e_{r,r+1}(\sigma(\bar{v}+z))\cD T(\bar{v}+z+\sigma(\varepsilon_{r,r+1}))\\
\end{array}$$

\begin{itemize}
\item[(i)] If $e_{r,r+1}(\sigma_{s}(\bar{v}+z))=0$ then, as in Case 1. that implies $\bar{v}_{r+1,t}-\bar{v}_{r,s} \in \mathbb{Z}$ for some $t$, with $r \geq k+1$, that is a contradiction.

\item[(ii)] If $e_{r,r+1}(\sigma_{s}(\bar{v}+z))\neq 0$ then, $\mathcal{D}T(\bar{v}+z)\prec \mathcal{D}T(\bar{v}+z+\delta^{rs})$ and as $c_{k2}(\mathcal{D}T(\bar{v}+z+\delta^{rs}))\prec T(\bar{v}+z+\delta^{rs})$, the coefficient of $T(\bar{v}+z+\delta^{rs})$ on the decomposition of $c_{k2}E_{r,r+1}\mathcal{D}T(\bar{v}+z)$ is not zero. Therefore, $\mathcal{D}T(\bar{v}+z)\prec T(\bar{v}+z+\delta^{rs})$, which contradicts the hypothesis.
\end{itemize}

Therefore, we have $\bar{v}_{rs}-\bar{v}_{r-1,t} \in \mathbb{Z}$ for some $(r,s,t)$ is satisfied, in this case, with  $r \geq k+1$, what is a contradiction. Thus the module $N$ is irreducible.
\end{proof}
\begin{remark}
Note that if $\Omega_k(Tab(w))=\emptyset$, Theorem \ref{Theorem: Basis for irreducible modules generalization} gives an alternative proof of Theorem 4.14  in \cite{FGR3} which gives a sufficient condition on the entries of $\bar{v}$ in order to have the irreducibility of $V(T(\bar{v}))$. 
\end{remark}

\section{Irreducibility of $V(T(\bar{v}))$}

Now we will prove that the conditions given in Theorem 4.14 in \cite{FGR3} are necessary for irreducibility of the module $V(T(\bar{v}))$. For this we need some definitions and lemmas.

\begin{definition}
For any $w\in\mathbb{Z}^{\frac{n(n-1)}{2}}$ we write 
\begin{align*}
\Omega^{+}(Tab(w))&:=\{(r,s,t) \  |  \  (\bar{v}_{rs}+w_{rs})-(\bar{v}_{r-1,t}+w_{r-1,t})\in \mathbb{Z}_{\geq 0}\}
\end{align*}
\end{definition}
\begin{remark}
Note that, in the case of generic modules, if $Tab(z)\prec Tab(w)$ implies $|\Omega^{+}(Tab(z))|\leq |\Omega^{+}(Tab(w))|$ (see Theorem \ref{Basis for irreducible generic modules gl(n)}(i)). However, for singular modules we can have $|\Omega^{+}(Tab(z))|-1=|\Omega^{+}(Tab(w))|$. In fact, consider $\bar{v}
=(a,b,c,x,x,x)$ such that $\{a-x, b-x, c-x\}\cap\mathbb{Z}=\emptyset$ and $w=(0,0,0)$, then $|\Omega^{+}(Tab(z))|=2$ while $E_{32}Tab(z)=Tab(z-\delta^{21})$ and $|\Omega^{+}(Tab(z-\delta^{21}))|=1$.
\end{remark}
The following lemma shows that $|\Omega^{+}(Tab(z))|-1$ is the infimum for the size of $\Omega^{+}(Tab(w))$ for any $Tab(w) \in U\cdot Tab(z)$

\begin{proposition}\label{size of Omega + decreases at most 1}
Let $Tab(z)$ and $Tab(w)$ be tableaux in $\mathcal{B}(T(\bar{v}))$ such that $Tab(w) \in U\cdot Tab(z)$, then $|\Omega^{+}(Tab(w))|\geq |\Omega^{+}(Tab(z))|-1$.
\end{proposition}

To prove this  proposition we will prove first that the lemma is true when $Tab(z)\prec_{g}Tab(w)$ for some $g\in\mathfrak{gl}(n)$ of the form $E_{r,r+1}$ or $E_{r+1,r}$.

\begin{lemma}\label{size of Omega + decreases at most 1 with E_{r,r+1}}
Let $Tab(z)$ and $Tab(w)$ be tableaux in $\mathcal{B}(T(\bar{v}))$ such that $Tab(z)\prec_{g}Tab(w)$ for some $g\in\mathfrak{gl}(n)$ of the form $E_{r,r+1}$ or $E_{r+1,r}$ with $1\leq r\leq n-1$, then $|\Omega^{+}(Tab(w))|\geq |\Omega^{+}(Tab(z))|-1$.
\end{lemma}
\begin{proof}
We will analyze the action of generators of $\mathfrak{gl}(n)$ of the form $E_{r,r+1}$ or $E_{r+1,r}$ in all tableaux $Tab(z) \in V(T(\bar{v}))$  such that $|\Omega^+(Tab(w)| \leq |\Omega^+(Tab(z)|-1$ and $Tab(z) \prec Tab(w)$. A case by case verification we have that, the coefficient of  tableau $Tab(w)$ is equal to zero whenever $|\Omega^+(Tab(w)| \leq |\Omega^+(Tab(z)|-2$ and the list of all possible tableaux $Tab(z)$ and $Tab(w)$ such that $|\Omega^+(Tab(w)| = |\Omega^+(Tab(z)|-1$ and the coefficient of $Tab(w)$ is not zero is the following:

\begin{enumerate}[(I)]
\item 
$\mathcal{D}\left(\begin{array}{ccc} x &  & x+a \\   & x  &  \end{array}\right) \prec_{E_{k-1,k}} \left(\begin{array}{ccc} x &  & x+a \\   & x+1  &  \end{array}\right) \ , \  a \in \mathbb{Z}_{<0} $.\vspace{0.2cm}

\item 
$\mathcal{D}\left(\begin{array}{ccc}  & x & \\  x &   & x+a \end{array}\right) \prec_{E_{k-1,k}} \left(\begin{array}{ccc}  & x &  \\ x+a  &   & x+1  \end{array}\right) \ , \ a \in \mathbb{Z}_{<0} $.\vspace{0.2cm}

\item $\mathcal{D}\left(\begin{array}{ccc} x &  & x+a \\   & x  &  \end{array}\right) \prec_{E_{k+1,k}} \left(\begin{array}{ccc} x-1 &  & x+a \\   & x  &  \end{array}\right) \ , \  a \in \mathbb{Z}_{<0} $. \vspace{0.2cm}

\item $\left(\begin{array}{ccc}  & x & \\  x &   & x \end{array}\right) \prec_{E_{k-1,k}} \left(\begin{array}{ccc}  & x &  \\ x  &   & x+1  \end{array}\right) $. \vspace{0.2cm}

\item $\left(\begin{array}{ccc} x &  & x \\   &  x &  \end{array}\right) \prec_{E_{k+1,k}} \left(\begin{array}{ccc} x-1 &  & x \\  &  x &  \end{array}\right) $.
\end{enumerate}
Where configurations above represent the part of the tableaux around row $k$.
 First of all, by Corollary \ref{elements of gl(n) acting by the classical GT formulas} is enough to consider $r\in\{k,k-1\}$, in fact, for the other cases the action is given by the classical Gelfand-Tsetlin formulas, as in the generic case. So we have $|\Omega^{+}(Tab(w))|\geq |\Omega^{+}(Tab(z))|$ (see Theorem \ref{Basis for irreducible generic modules gl(n)}(i)).\\
Assume first that $g=E_{r,r+1}$. The action of $E_{r,r+1}$ on basis elements of $V(T(\bar{v}))$ is given by.

\begin{equation}\label{formulas for E_{r,r+1} on T}
\begin{split}
E_{r,r+1}T(\bar{v}+w)=&\displaystyle\sum_{\sigma}\mathcal{D}^{\bar{v}}\left((v_{ri}-v_{rj})e_{r,r+1}(\sigma(v+w)\right)T(\bar{v}+w+\sigma(\delta^{r1}))+\\
&\displaystyle\sum_{\sigma}\left((v_{ri}-v_{rj})e_{r,r+1}(\sigma(v+w))\right)(\bar{v})\mathcal{D}T(\bar{v}+w+\sigma(\delta^{r1})),
\end{split}
\end{equation}

\begin{equation}\label{formulas for E_{r,r+1} on DT}
\begin{split}
E_{r,r+1}\mathcal{D}T(\bar{v}+w)=&\displaystyle\sum_{\sigma}\mathcal{D}^{\bar{v}}\left(e_{r,r+1}(\sigma(v+w)\right)T(\bar{v}+w+\sigma(\delta^{r1}))\\
&+\displaystyle\sum_{\sigma}e_{r,r+1}(\sigma(\bar{v}+w))\mathcal{D}T(\bar{v}+w+\sigma(\delta^{r1})).
\end{split}
\end{equation}
Depending on the $Tab(z)$ and $Tab(z+\delta^{r1})$ being regular tableau or derivative tableau, we have to look at different coefficients as shows the following table:
 \begin{center}
\begin{tabular}{|c|c|c|c|}
\hline
Type & $Tab(z)$ & $Tab(z+\delta^{r1})$ & Coefficient of $Tab(z+\delta^{r1})$\\ \hline
(a)& $T(\bar{v}+z)$ & $T(\bar{v}+z+\delta^{r1})$ & $\mathcal{D}^{\bar{v}}\left((v_{ki}-v_{kj})e_{r,r+1}(v+z)\right)$\\ \hline
(b)&$\cD T(\bar{v}+z)$ & $\cD T(\bar{v}+z+\delta^{r1})$ & $e_{r,r+1}(\bar{v}+z)$\\ \hline
(c)&$T(\bar{v}+z)$ & $\cD T(\bar{v}+z+\delta^{r1})$ & $\left((v_{ki}-v_{kj})e_{r,r+1}(v+z)\right)(\bar{v})$\\ \hline
(d)&$\cD T(\bar{v}+z)$ & $T(\bar{v}+z+\delta^{r1})$ & $\mathcal{D}^{\bar{v}}\left(e_{r,r+1}(v+z)\right)$\\ \hline
\end{tabular}
\end{center}
\begin{itemize}
\item[(i)] Consider a tableaux $Tab(z)$ such that $(\bar{v}+z)_{ki}=(\bar{v}+z)_{k-1,t}=x$ and $(\bar{v}+z)_{kj}=x+a$, with $a \in \mathbb{Z}$. We will represent this part of tableau $Tab(z)$  by (the row where appear two  variable equal to $x$ is the  $k$-th row of the tableau $Tab(z)$).
$$Tab(z)=\left(\begin{array}{ccc} x & & x+a \\  &  x & \end{array}\right)$$ 
This tableau can be regular (if $a \geq 0$) or derivative (if $a<0$). We will analyse this two cases. When $E_{k-1,k}$ acts in this tableau we obtain the following tableau:
$$Tab(z+\delta^{k-1,t})=\left(\begin{array}{ccc} x & & x+a \\  &  x+1 & \end{array}\right)$$
Note that 
$$
|\Omega^{+}(Tab(z+\delta^{k-1,t}))|=\begin{cases}
|\Omega^{+}(Tab(z))|-2, & \text{if} \ \ a=0\\
|\Omega^{+}(Tab(z))|-1, & \text{otherwise.}
\end{cases} 
 $$

For $a \geq 0$ the tableaux $Tab(z)$ is regular. In this case we will analyse the coefficients of types (a) and (c). Recall that,
\begin{equation}\label{base1}
e_{k-1,k}(v+z)=-\displaystyle\frac{\prod_{t=1}^k\left((v+z)_{k-1,1}-(v+z)_{kt}\right)}{\prod_{t \neq 1}^{k-1}\left((v+z)_{k-1,1}-(v+z)_{k-1,t}\right)}
\end{equation}
In this case, $e_{k-1,k}(v+z)$ is smooth function (because the singularity is in $k$-th row and in the denominator appear differences between elements of the  $(k-1)$-th row). Thus, by Lemma (\ref{Properties of D on functions}) follows that:
\begin{itemize}
\item[$\bullet$] $\mathcal{D}^{\bar{v}}\left((v_{ri}-v_{rj})e_{k-1,k}(v+z)\right)=e_{k-1,k}(\bar{v}+z)$, and, as in this case we have the relation  $(\bar{v}+z)_{kj}=(\bar{v}+z)_{k-1,t}$, therefore $e_{k-1,k}(\bar{v}+z)=0$.
\item[$\bullet$] $\left((v_{ki}-v_{kj})e_{k-1,k}(v+z)\right)(\bar{v})=0.$
\end{itemize}
If $a<0$, the tableau $Tab(z)$ is derivative tableau. In this case, we need analyze the coefficients of types (b) and (d). Now the coefficients of tableau $Tab(z+\delta^{k-1,t})$ are $e_{k-1,k}(\bar{v}+z)$ and $\mathcal{D}^{\bar{v}}\left(e_{k-1,k}(v+z)\right)$. In this case, we have that
\begin{itemize}
\item[$\bullet$] $e_{k-1,k}(\bar{v}+z)=0$, because in the numerator of this rational function appear a difference $(\bar{v}+z)_{kj}-(\bar{v}+z)_{k-1,t}$ that is equal to zero, in this case.
\item[$\bullet$] For the coefficient $\mathcal{D}^{\bar{v}}\left(e_{k-1,k}(\bar{v}+z)\right)$, we have the following:
$$\begin{array}{rl}
\mathcal{D}^{\bar{v}}\left(e_{k-1,k}(\bar{v}+z)\right)=&\mathcal{D}^{\bar{v}}(-(v_{k-1,t}-(v_{ki}+a))(v_{k-1,t}-v_{ki})\varphi(v))\\
=&-\frac{1}{2}\varphi(\bar{v})a
\end{array}$$
\end{itemize}
where $\varphi(v)$ is a rational function of the entries of the vector $v$ that not depends of the entries $v_{ki}$ end $v_{kj}$, moreover $\varphi(\bar{v})\neq 0$. As $a \neq 0$, in this case we have 
$\mathcal{D}^{\bar{v}}\left(e_{k-1,k}(\bar{v}+z)\right)\neq 0$. 

\item[(ii)]
Now we will consider a tableau $Tab(z)$ such that $k \neq n-1$, $(\bar{v}+z)_{k+1,t}=(\bar{v}+z)_{ki}=x$ and $(\bar{v}+z)_{kj}=x+a$, with $ a \in \mathbb{Z}$. A representation of part this tableau is
$$Tab(z)=\left(\begin{array}{ccc}  & x &  \\ x  &   &x+a  \end{array}\right)$$
This tableau should be regular or derivative depending on the value of $a$. We will analyze this two cases. When $E_{k,k+1}$ acts in this tableau we obtain the tableau
$$Tab(z+\delta^{k,t})=\left(\begin{array}{ccc}  & x &  \\ x+1 &  & x+a\end{array}\right)$$
In this case, we have that $|\Omega^+Tab(z+\delta^{k,t})|=|\Omega^+Tab(z)|-1$. Initially we will assume that this tableau is regular tableau ($a \geq 0$) and we will analyze the coefficients of type (a) and (c). For this, recall that:
\begin{equation}\label{base2}
e_{k,k+1}(v+z)=-\displaystyle\frac{\prod_{t=1}^{k+1}\left((v+z)_{k,1}-(v+z)_{k+1,t}\right)}{\prod_{t \neq 1}^{k}\left((v+z)_{k1}-(v+z)_{kt}\right)}.
\end{equation}
If $a>0$, in this case that $e_{k,k+1}(v+z)$ is a smooth function, then by  Lemma \ref{Properties of D on functions} follows that
\begin{itemize}
\item[$\bullet$] $\mathcal{D}^{\bar{v}}\left((v_{ri}-v_{rj})e_{k,k+1}(v+z)\right)=e_{k,k+1}(\bar{v}+z)$, and, as we have the relation  $(\bar{v}+z)_{k+1,t}=(\bar{v}+z)_{k-1,1}$ then, $e_{k,k+1}(\bar{v}+z)=0$.
\item[$\bullet$] $\left((v_{ki}-v_{kj})e_{k-1,k}(v+z)\right)(\bar{v})=0.$
\end{itemize}
On the other hand, if $a=0$, we have
$$\begin{array}{rl}\left((v_{k1}-v_{kj})e_{k-1,k}(v+z)\right)(\bar{v})=&\left(\displaystyle\frac{\prod_{t=1}^k\left((v+z)_{k-1,1}-(v+z)_{kt}\right)}{\prod_{t \neq 1,j}^{k-1}\left((v+z)_{k-1,1}-(v+z)_{k-1,t}\right)}\right)(\bar{v})\\
=0
\end{array}$$

$$\begin{array}{rl}\mathcal{D}^{\bar{v}}\left((v_{ki}-v_{kj})e_{k-1,k}(v+z)\right)=&\mathcal{D}^{\bar{v}}\left(\displaystyle\frac{\prod_{t=1}^k\left((v+z)_{k-1,1}-(v+z)_{kt}\right)}{\prod_{t \neq 1,j}^{k-1}\left((v+z)_{k-1,1}-(v+z)_{k-1,t}\right)}\right)\\
=0&
\end{array}$$

Finally, if $a<0$ the tableau $Tab(\bar{v}+z)$ is a derivative tableau. In this case we will analyze the coefficients of type (b) and (d). In this case, using the formula (\ref{base2}), follows that:
\begin{itemize}
\item[$\bullet$] $e_{k-1,k}(\bar{v}+z)=0$, because in the numerator of this rational function appear the difference $(\bar{v}+z)_{k1}-(\bar{v}+z)_{k+1,1}$, that is equal to zero in this case.
\item[$\bullet$] More one time using the formula (\ref{base2}), we have that $\mathcal{D}^{\bar{v}}\left(e_{k-1,k}(\bar{v}+z)\right)\neq 0$.
\end{itemize}
Continuing this analysis, case by case,  we can identify all tableaux $Tab(w) \in U\cdot Tab(z)$  such that the correspondent coefficient is no zero is a tableaux $Tab(w)$ such that $|\Omega^{+}(Tab(w))|=|\Omega^{+}(Tab(z))|-1$, that was described in begin the poof of Lemma \ref{size of Omega + decreases at most 1 with E_{r,r+1}} (cases (I)-(V)). Moreover, for all tableaux $Tab(w) \in V(T(\bar{V}))$ such that
$|\Omega^{+}(Tab(w))| \leq |\Omega^{+}(Tab(z))|-2$ have respective coefficients  equal to zero.

\end{itemize}
\end{proof}


The following lemma shows that, for each of the cases described in Lemma \ref{size of Omega + decreases at most 1 with E_{r,r+1}} where $Tab(z) \prec Tab(w)$ and $|\Omega^+(Tab(w)| = |\Omega^+(Tab(z)|-1$, (cases (I)-(V)), the action of the basis elements of $\mathfrak{gl}(n)$ on $Tab(w)$ will generate tableaux $Tab(w')$ such that $|\Omega^+(Tab(w))|\leq |\Omega^+(Tab(w'))|$ (i.e. $Tab(w)\prec_{g} Tab(w')$ for some $g\in\mathfrak{gl}(n)$ of the form $E_{r,r+1}$ or $E_{r+1, r}$ implies  $|\Omega^+(Tab(w))|\leq |\Omega^+(Tab(w'))|$).
\begin{lemma}\label{tableaux that size decrease one}
Let $Tab(w)$ be a tableau such that $Tab(z) \prec Tab(w)$ and $|\Omega^+(Tab(w))|=|\Omega^+(Tab(z))|-1$ for some $Tab(z)$. If $Tab(w)\prec_{g} Tab(w')$ for some $g\in\mathfrak{gl}(n)$ of the form $E_{r,r+1}$ or $E_{r+1, r}$, then  $|\Omega^+(Tab(w))|\leq |\Omega^+(Tab(w'))|$.
\end{lemma}
\begin{proof}
By Lemma \ref{size of Omega + decreases at most 1 with E_{r,r+1}} the set of all possible tableaux $Tab(w)$ satisfying the condition is given by:
\begin{enumerate}[(1)]
\item 
$\left(\begin{array}{ccc} x &  & x+a \\   & x+1  &  \end{array}\right); \ \ \ \left(\begin{array}{ccc}  & x &  \\ x+a  &   & x+1  \end{array}\right) ; \ \ \left(\begin{array}{ccc} x-1 &  & x+a \\   & x  &  \end{array}\right)$ \vspace{0.2cm}\\
with $a \in \mathbb{Z}_{<0}$ \vspace{0.2cm}.
\item $\left(\begin{array}{ccc}  & x &  \\ x  &   & x+1  \end{array}\right);\ \ \left(\begin{array}{ccc} x-1 &  & x \\  &  x &  \end{array}\right) $.
\end{enumerate}
Now, for each $Tab(w)$ as before we will construct subsets $W(Tab(w))$ and $W^*(Tab(w))$, of $\mathcal{B}(T(\bar{v}))$ such that: 
\begin{itemize}
\item[(i)] $Tab(w)\in W(Tab(w))\cup W^{*}(Tab(w))$ and $Tab(w')\in W(Tab(w))\cup W^{*}(Tab(w))$ implies $|\Omega^+(Tab(w'))|\geq |\Omega^+(Tab(w))|$.
\item[(ii)] $W(Tab(w)) \cup W^*(Tab(w))$ is $\mathfrak{gl}(n)$-invariant.
\end{itemize}
The construction before will be enough to finish the proof. In fact, if $Tab(w)\prec_{g} Tab(w')$ with $g\in\mathfrak{gl}(n)$, then $Tab(w')\in W(Tab(w)) \cup W^*(Tab(w))$ (because of condition (ii)) and, then $|\Omega^+(Tab(w'))|\geq |\Omega^+(Tab(w))|$ (because of condition (i)).


For instance, consider $Tab(w)$ to be of the form $\left(\begin{array}{ccc} x &  & x+a \\   & x+1  &  \end{array}\right).$ Define the following subsets of $\mathcal{B}(T(\bar{v}))$:
$${\tiny W(Tab(w))=\left\{ \left( \begin{array}{ccc} x+b &  & x+a+c \\   & x+1+d  &  \end{array}\right) \ | \ b-1-d<0\ \mbox{ and }  \ b+a-1-d<0\right\}}$$
$${\tiny W^*(Tab(w))=\left\{\left(\begin{array}{ccc} x+b &  & x+a+c \\   & x+1+d  &  \end{array}\right) \ | \ b-1-d \geq 0\ \mbox{ or }  \ b+a-1-d \geq 0\right\}.}$$
Note that, the tableaux in $W(Tab(w))$ or $W^{*}(Tab(w))$ can be regular or derivative. It is easy to check that $W(Tab(w))\cup W^{*}(Tab(w))$ is $\mathfrak{gl}(n)$-invariant. Also, $Tab(w') \in W(Tab(w))\cup W^{*}(Tab(w))$ satisfies $|\Omega^+(Tab(w'))|\geq |\Omega^+(Tab(w))|$. In fact, if $Tab(w') \in W(Tab(w))$ follows that  $|\Omega^+(Tab(w'))|=|\Omega^+(Tab(w))|$ and, by on the other hand, if $Tab(w') \in W^*(Tab(w))$ follows that  $|\Omega^+(Tab(w'))|>|\Omega^+(Tab(w))|$.

The construction of $W(Tab(w))$ and $W^*(Tab(w))$ for the other cases is analogous.

\end{proof}

{\bf Proof of Lemma \ref{size of Omega + decreases at most 1}.} We will show that if $Tab(w) \in E_{1}\cdots E_{t}\cdot Tab(z)$, then \begin{equation}\label{inequality}
|\Omega^{+}(Tab(w))| \geq |\Omega^{+}(Tab(z))|-1,
\end{equation}
and by linearity the result follows for every element of $u \in U(\mathfrak{gl}(n))$.

\begin{proof}
Let $Tab(z) \in V(T(\bar{v}))$ and $E_{1}\cdots E_{t} \in U$. We will use induction over $t$. Indeed, if $t=1$, this result is true by Lemma \ref{size of Omega + decreases at most 1 with E_{r,r+1}}. Now, suppose that the lemma is true for every $s$ such that $1 \leq s \leq t-1$. We will proof that the result is true for $t$. Indeed, as 
$$E_{1}E_{2}\cdots E_{t-1}E_{t}\cdot Tab(z)=E_{1}E_{2}\cdots E_{t-1}(E_{t}\cdot Tab(z))$$
we have to consider two cases:\vspace{0,2cm}\\
{\it Case 1.} If $Tab(z)$ is not a tableau described in Lemma \ref{size of Omega + decreases at most 1 with E_{r,r+1}} (cases (I)-(V)), follows that  $E_{t}\cdot Tab(z)=\sum_i\alpha_iTab(w_i)$, where $|\Omega^{+}(Tab(w_i))| \geq |\Omega^{+}(Tab(z))|$. In this case, by induction hypothesis, the action of $E_{1}E_{2}\cdots E_{t-1}$ over $(E_{t}\cdot Tab(w))$ we will obtain a linear combination of tableaux $Tab(w''_i)$ such that 
$|\Omega^{+}(Tab(w''_i))| \geq |\Omega^{+}(Tab(z))|-1$
which proves the  desired result in this case.\vspace{0,2cm}\\
{\it Case 2.} Now we will assume that $Tab(z)$ is a tableau described in Lemma \ref{size of Omega + decreases at most 1 with E_{r,r+1}} (cases (I)-(V)). In this case follows that
$$E_{t}\cdot Tab(z)=\sum_i\alpha_1Tab(w_i)+\sum_i\beta_1Tab(w'_i)$$
where $|\Omega^{+}(Tab(w_i))| \geq |\Omega^{+}(Tab(z))|$ and
$|\Omega^{+}(Tab(w'_i))| = |\Omega^{+}(Tab(z))|-1$
Thus, by induction hypothesis, follows that when $E_{1}E_{2}\cdots E_{t-1}$ acts in tableaux $Tab(w_i)$ we obtain a linear combination of new tableaux $Tab(w''_i)$ such that $|\Omega^{+}(Tab(w''_i))| \geq |\Omega^{+}(Tab(z))|-1$. By other hand, when $E_1 \cdots E_{t-1}$ acts in tableaux $Tab(w'_i)$, we will get a linear combination of tableaux belong the set $W(Tab(w))\cup W^*(Tab(w))$ (that was described in Lemma \ref{tableaux that size decrease one}). Thus, in this case we have that in decomposition $E_1 \cdots E_{t-1}\cdot Tab(w'_i)$ will appear tableaux $Tab(w_i^*)$ such that
$$|\Omega^{+}(Tab(w^*_i))| = |\Omega^{+}(Tab(z))|-1$$
which prove the result in this second case.
\end{proof}


\begin{corollary}\label{proper}
If $z\in\mathbb{Z}^{\frac{n(n-1)}{2}}$ is such that $|\Omega^{+}(Tab(z))|\geq 1$ then,  
$U\cdot Tab(z)$ is a proper submodule of $V(T(\bar{v}))$.
\end{corollary}

\begin{proof}
Indeed, if the module $U \cdot Tab(z)$ is not a proper submodule of $V(T(\bar{v}))$, follows that for every $Tab(w) \in V(T(\bar{v}))$ we have $Tab(z) \prec Tab(w)$. But by Lemma(\ref{size of Omega + decreases at most 1}), follows that $|\Omega^{+}(Tab(w))| \geq |\Omega^{+}(Tab(z))|-1$. We have two cases to consider:\\
{\it Case 1.}
If $|\Omega^{+}(Tab(z))| \geq 2$, follows that
$$|\Omega^{+}(Tab(w))| \geq |\Omega^{+}(Tab(z))|-1\geq 2-1=1$$
but, this is a contradiction, we always have $T(w) \in V(T(\bar{v}))$ such that $|\Omega^{+}(Tab(w))|=0$.\\
{\it Case 2.}
If $|\Omega^{+}(Tab(z))|=1$, follows that the unique integer difference $\bar{v}_{rs}-\bar{v}_{r-1,t}$ is not close to critical line, because if this integer difference is close to critical line we would  $|\Omega^{+}(Tab(z))| \geq 2$. But away from the critical  line we have the generic case, where we have the following inequality
$$|\Omega^{+}(Tab(w))| \geq |\Omega^{+}(Tab(z))|=1$$
As in the last case this is a contradiction.\\

Hence, if $|\Omega^{+}(Tab(z))|\geq 1$ the module $U\cdot Tab(z)$ could not be  
the full module $V(T(\bar{v}))$, thus in this case, $U \cdot Tab(z)$ is a proper submodule.
\end{proof}
 



\begin{theorem}\label{sufficient conditions for L to be irred}
If $V(T(\bar{v}))$ is irreducible then, $\bar{v}_{rs}-\bar{v}_{r-1,t}\notin \mathbb{Z}$ for any $1\leq s \leq r\leq n$ and  $1\leq t\leq r-1$. 
\end{theorem}
\begin{proof}
Suppose that $\bar{v}_{rs}-\bar{v}_{r-1,t} \in \mathbb{Z}$ for some $1 \leq t < r \leq n, 1 \leq s \leq r$, choose $w \in \mathbb{Z}^{\frac{n(n-1)}{2}}$ such that $\left|\Omega^+(Tab(w))\right|$ is maximal ($|\Omega^+(Tab(w))|\geq 1$) and consider the submodule  $N$ of $ V(T(\bar{v}))$ generate by $Tab(w)$. Thus by \ref{proper} follows that the submodule $N$ is proper. 
\end{proof}


\end{document}